\documentclass[10pt, reqno]{amsart}
\usepackage{amscd, amssymb, amsmath}
\usepackage{hyperref}
\usepackage{epsfig}
 \usepackage{verbatim}
\usepackage{dsfont}
\usepackage{todonotes}
\usepackage{hyperref}

\newtheorem{Df}{Definition}
\newtheorem{theorem}[Df]{Theorem}

\newtheorem{prop}[Df]{Proposition}

\newtheorem{corollary}[Df]{Corollary}
\newtheorem{conj}[Df]{Conjecture}
\numberwithin{Df}{subsection}
\numberwithin{equation}{subsection}

\newcommand{\coev}{\ensuremath{\operatorname{coev}} }
\newcommand{\ev}{\ensuremath{\operatorname{ev}} }

\newcommand{\C}{\ensuremath{\mathbb{C}}}
\newcommand{\Z}{\ensuremath{\mathbb{Z}}}

\newcommand{\gl}{\ensuremath{\mathfrak{gl}}}

\newcommand{\ideal}{I}

\newcommand{\obj}{Ob}
\newcommand{\End}{\operatorname{End}}

\newcommand{\tr}{\operatorname{tr}}
\newcommand{\Id}{\operatorname{Id}}

\newcommand{\osp}{\mathfrak{osp}}

\newcommand{\mt}{\operatorname{\mathsf{t}}}
\newcommand{\md}{\operatorname{\mathsf{d}}}

\newcommand{\Fcat}{\mathcal{F}}

\newcommand{\Proj}{\ensuremath{\mathcal{P}roj}}
\newcommand{\fg}{\ensuremath{\mathfrak{g}}}

\newcommand{\fh}{\ensuremath{\mathfrak{h}}}
\newcommand{\fb}{\ensuremath{\mathfrak{b}}}
\newcommand{\fe}{\ensuremath{\mathfrak{e}}}
\newcommand{\fl}{\ensuremath{\mathfrak{l}}}

\newcommand{\fa}{\ensuremath{\mathfrak{a}}}

\newcommand{\0}{\ensuremath{\bar{0}}}
\newcommand{\1}{\ensuremath{\bar{1}}}

\newcommand{\sdim}{\ensuremath{\operatorname{sdim}}}
\newcommand{\defect}{\ensuremath{\operatorname{def}}}
\newcommand{\atyp}{\ensuremath{\operatorname{atyp}}}
\newcommand{\rank}{\operatorname{rank}}
\newcommand{\V}{\mathcal{V}}
\newcommand{\HH}{\operatorname{H}}
\newcommand{\res}{\operatorname{res}}
\newcommand{\Res}{\operatorname{Res}}

\newcommand{\Xvar}{\mathcal{X}}

\usepackage[capitalise]{cleveref}
\crefname{theorem}{Theorem}{Theorems}
\crefname{fact}{Fact}{Facts}
\crefname{note}{Note}{Notes}
\crefname{lemma}{Lemma}{Lemmas}
\crefname{alg}{Algorithm}{Algorithms}
\crefname{remark}{Remark}{Remarks}
\crefname{example}{Example}{Examples}
\crefname{prop}{Proposition}{Propositions}
\crefname{conj}{Conjecture}{Conjectures}
\crefname{cor}{Corollary}{Corollaries}
\crefname{defn}{Definition}{Definitions}
\crefname{equation}{}{}

\begin{document}
\title[{The generalized Kac-Wakimoto conjecture and support varieties for $\mathfrak{osp}(m|2n)$}]{The generalized Kac-Wakimoto conjecture and support varieties for the Lie superalgebra $\mathfrak{osp}(m|2n)$}

\author{Jonathan Kujawa}
\address{Mathematics Department\\
University of Oklahoma\\
Norman, OK 73019}
\thanks{Research of the author was partially supported by NSF grant
DMS-0734226 and NSA grant H98230-11-1-0127.}\
\email{kujawa@math.ou.edu}
\date{\today}

\begin{abstract}  Atypicality is a fundamental combinatorial invariant for simple supermodules of a basic Lie superalgebra.   Boe, Nakano, and the author gave a conjectural geometric interpretation of atypicality via support varieties. Inspired by low dimensional topology, Geer, Patureau-Mirand, and the author gave a generalization of the Kac-Wakimoto atypicality conjecture.  We prove both of these conjectures for the Lie superalgebra $\osp(m|2n)$.
 \end{abstract}

\maketitle
\setcounter{tocdepth}{1}

\section{Introduction}\label{S:Intro}   

\subsection{}  Let $\fg=\fg_{\0}\oplus \fg_{\1}$ be a basic classical Lie superalgebra over the complex numbers.  An important category of $\fg$-supermodules is the 
category $\Fcat$ of finite dimensional integrable $\fg$-supermodules.   Starting with the work of Kac \cite{Kac1, Kac2}, the category $\Fcat$ has been the object of investigation for more than 30 years by numerous researchers.  Of particular interest is the simple supermodules in $\Fcat$.   Most efforts have focused on obtaining character formulas (to mention only a few of the more prominent papers in the area, see \cite{brundan, CLW,GS, serganovaICM}).

Recently two new lines of investigation have developed.  The category $\Fcat$ shares a number of features with the modular representations of finite groups and, more generally, finite group schemes.  For example, $\Fcat$ is not semisimple, has enough projectives, and usually projectives and injectives in $\Fcat$ coincide.  Motived by the successful use of cohomology and support varieties in the finite group scheme setting, the authors of \cite{BKN1} began an investigation of $\Fcat$ using analogous tools.  They conjectured that support varieties provide a geometric interpretation of the combinatorial invariant known as atypicality.  Namely, given a supermodule $M$ in $\Fcat$, let $\V_{(\fg, \fg_{\0})}(M)$ denote the support variety associated to $M$ as in \cite{BKN1}.   Given a simple $\fg$-supermodule $L(\lambda)$ of highest weight $\lambda$, let $\atyp (\lambda)$ denote the atypicality of $\lambda$.  Precise definitions can be found in the body of this paper. The following ``atypicality conjecture'' is given in \cite[Conjecture 7.2.1]{BKN1}\footnote{More accurately, there the support variety of the detecting algebra is used but conjecturally the dimension of that variety coincides with the one used here.}.

\begin{conj}\label{C:atypconjintro}  Let $\fg$ be a basic classical Lie superalgebra and let $L(\lambda)$ be a simple supermodule in $\Fcat$.  Then,
\[
\dim \V_{(\fg , \fg_{\0})}(L(\lambda)) = \atyp (\lambda).
\]  
\end{conj}
\noindent This conjecture was proven for $\gl(m|n)$ in \cite{BKN2}.

In a different direction, the authors of \cite{GKP} were motivated by questions in low dimensional topology to introduce modified trace and dimension functions for $\Fcat$.   Given a supermodule $M = M_{\0}\oplus M_{\1}$ in $\Fcat$ the appropriate analogue of dimension is superdimension:
\[
\sdim (M) := \dim M_{\0}- \dim M_{\1}.
\]  Plainly the superdimension can equal zero.  Let $\defect (\fg )$ denote the maximum possible value of $\atyp (\lambda)$ as $\lambda$ ranges over the highest weights of simple supermodules in $\Fcat$.  The following conjecture of Kac and Wakimoto \cite[Conjecture 3.1]{KW} makes precise when the superdimension of a simple supermodule vanishes.
\begin{conj}  Let $L(\lambda)$ be a simple $\fg$-supermodule in $\Fcat$, then 
\[
\atyp (\lambda) = \defect (\fg ) \text{ if and only if } \sdim (L(\lambda)) \neq 0.
\]
\end{conj} The authors of \cite{GKP} introduce modified dimension functions for $\Fcat$ and prove that they are a natural replacement for the superdimension when the superdimension vanishes.  In particular, they provide the following generalization of the Kac-Wakimoto conjecture \cite[Conjecture 6.3.2]{GKP}.  Given a supermodule $M$, let $\ideal_{M}$ denote the full subcategory of all supermodules which appear as a direct summand of $M \otimes X$ for some supermodule $X$ in $\Fcat$.
\begin{conj}  Let $\fg$ be a basic classical Lie superalgebra and let $L(\lambda)$ be a simple $\fg$-supermodule.  Then $L(\lambda)$ admits an ambidextrous trace and we can define a modified dimension function $\md_{L(\lambda)}$ on $\ideal_{L(\lambda)}$.  If $L(\mu)$ is another simple supermodule with 
\[
\atyp (\mu) \leq \atyp(\lambda),
\] then $L(\mu)$ is an object of $\ideal_{L(\lambda)}$, and
\[
\atyp (\mu) = \atyp (\lambda ) \text{ if and only if } \md_{L(\lambda)} (L(\mu)) \neq 0.
\]
\end{conj}  \noindent If $L(\lambda)$ is the trivial supermodule, then $\ideal_{L(\lambda)}=\Fcat$, $\atyp (\lambda) =\defect (\fg )$, and $\md_{L(\lambda)}=\sdim$.  In this way the generalized Kac-Wakimoto conjecture specializes to the ordinary Kac-Wakimoto conjecture.  Recently Serganova proved the ordinary Kac-Wakimoto conjecture for $\gl (m|n)$ and $\osp (m|2n)$ and the generalized Kac-Wakimoto conjecture for $\gl (m|n)$ \cite{serganova4}.

\subsection{}  In the present paper we consider the case when $\fg$ equals $\gl (m|n)$ or $\osp (m|2n)$.  That is, the Lie superalgebras of type ABCD in the Kac classification \cite{Kac1}.  Taken together these are the infinite families of basic classical Lie superalgebras in the Kac classification.  The results we prove are new for $\osp (m|2n)$ but, as we described above, are known for $\gl (m|n)$ by the work of Serganova \cite{serganova4} and Boe, Kujawa, and Nakano \cite{BKN2}.   However the proofs work equally well for $\gl (m|n)$ so we include them.  Perhaps the most interesting case which still remains is the type Q Lie superalgebras.  Although not basic, they have a notion of atypicality and the geometric and topological viewpoints apply.

In Section~\ref{S:GKW} we prove the generalized Kac-Wakimoto conjecture for $\Fcat$.  It is worth remarking that this has the following purely representation theoretic corollary.  Let $L(\lambda)$ and $L(\mu)$ be simple $\fg$-supermodules with $\atyp (\lambda) = \atyp (\mu)$.  Then there is a supermodule $X$ in $\Fcat$ such that $L(\lambda)$ is a direct summand of $L(\mu) \otimes X$. This in turn implies the support variety of all simple supermodules of the same atypicality coincide (see Theorem~\ref{T:constantatypicality}).  Similarly the complexity of the simple supermodules of the same atypicality coincide.  The recent calculation of complexity for the simple supermodules for $\gl (m|n)$ in \cite{BKN4} depends crucially on this result.

In Section~\ref{S:supports} we compute the support varieties of the simple $\fg$-supermodules in $\Fcat$. We show that if $L(\lambda)$ is a simple $\fg$-supermodule, then 
\[
\V_{(\fg , \fg_{\0})}(L(\lambda)) \cong \mathbb{A}^{\atyp (\lambda)}.
\] In particular this verifies Conjecture~\ref{C:atypconjintro}.  Note that the support variety is canonically defined for any object of $\Fcat$.  Thus the above result justifies the definition of the atypicality of a general supermodule via 
\[
\atyp (M) = \dim \V_{(\fg , \fg_{\0})}(M). 
\]

\subsection{Acknowledgements}  The author would like recognize his collaborations with Boe and Nakano, and Geer and Patureau-Mirand.  The results of this paper are a direct outgrowth of that work.  The author is thankful for many stimulating conversations.  The author would also like to acknowledge Serganova for helpful discussions.  In particular, several key ingredients used here were developed by Serganova, Duflo-Serganova, and Gruson-Serganova.  Finally, the author would like to recognize the stimulating environment provided by the Southeast Lie Theory conference series.  

\section{Preliminaries}\label{S:Prelims}  

\subsection{}  All vector spaces will be over the complex numbers, $\C $, and finite dimensional unless otherwise stated.  In most cases the vector spaces will have a $\Z_{2}$-grading, $V=V_{\0}\oplus V_{\1}$, and we will write $\bar{v} \in \Z_{2}$ for the degree of a homogeneous element $v \in V$.  We call an element $v \in V$ \emph{even} (resp.\ \emph{odd}) if $\bar{v}=\0$ (resp.\ $\bar{v}=\1$).

Let $\fg=\fg_{\0} \oplus \fg_{\1}$ denote one of the Lie superalgebras $\gl (m|n)$, $\osp (2m|2n)$, and $\osp (2m+1|2n)$ as defined in \cite{Kac1}.  In each case $\fg_{\0}$ is reductive as a Lie algebra and is \emph{classical} in the sense of \cite{BKN1}.  Furthermore, in each case we may define a bilinear form $(\; , \; ): \fg  \otimes \fg  \to \C $ by $(x,y) = \operatorname{str}(xy)$, where $\operatorname{str}$ is the supertrace.  This defines a nondegenerate, supersymmetric, invariant, even bilinear form and so by definition $\fg$ is \emph{basic}. 

Fix a choice of Cartan subalgebra $\fh \subset \fg_{\0}$ as in \cite{GS}.   The bilinear form on $\fg$ induces a bilinear form on the dual of the Cartan subalgebra, $\fh^{*}$, which we again denote by $(\; , \; )$.  In particular, we may choose a basis for $\fh^*$, $\varepsilon_{1}, \dotsc , \varepsilon_{m}, \delta_{1}, \dotsc , \delta_{n}$, on which 
\[
(\varepsilon_{i}, \varepsilon_{j}) = \delta_{i,j}, \hspace{.25in} (\varepsilon_{i}, \delta_{j}) = 0, \hspace{.25in} (\delta_{i}, \delta_{j}) = - \delta_{i,j}.
\]  With respect to our choice of Cartan subalgebra we have a decomposition of $\fg$ into root spaces. Each root space is one dimensional and spanned by a homogenous vector.  Consequently, we may define the parity of a root to be the parity of the corresponding root space.  We write $\Phi$ (resp.\ $\Phi_{\0}$ and $\Phi_{\1}$) for the set of roots (resp.\ set of even and odd roots).

We can explicitly describe the root systems as follows.   If $\fg = \gl (m|n)$, then the roots are 
\begin{align*}
\Phi_{\0} &= \left\{\varepsilon_{i}-\varepsilon_{j} \mid i \neq j \right\} \cup  \left\{\delta_{i}-\delta_{j} \mid   i \neq j  \right\}, \\
\Phi_{\1} &= \left\{\pm (\varepsilon_{i}-\delta_{j})  \right\}.
\end{align*} If $\fg = \osp (2m|2n)$, then the roots are
\begin{align*}
\Phi_{\0} &= \left\{\pm \varepsilon_{i}\pm \varepsilon_{j} \mid  i \neq j  \right\} \cup  \left\{\pm \delta_{i}\pm \delta_{j} \mid i \neq j  \right\} \cup \left\{2\delta_{i}  \right\}, \\
\Phi_{\1} &= \left\{\pm \varepsilon_{i}\pm \delta_{j}  \right\}.
\end{align*} If $\fg = \osp (2m+1|2n)$, then the roots are
\begin{align*}
\Phi_{\0} &= \left\{\pm \varepsilon_{i}\pm \varepsilon_{j} \mid   i \neq j  \right\} \cup  \left\{\pm \delta_{i}\pm \delta_{j} \mid  i \neq j \right\} \cup \left\{\pm \varepsilon_{i} \right\} \cup \left\{2\delta_{i}  \right\}, \\
\Phi_{\1} &= \left\{\pm \varepsilon_{i}\pm \delta_{j}  \right\}.
\end{align*}  In each case the subscripts on the epsilons are from among $1, \dotsc , m$ and the subscripts on the deltas are from among $1, \dotsc , n$.

We call a finite dimensional $\fg$-supermodule \emph{integrable} if all its weights lie in the $\Z$-span of $\varepsilon_{1}, \dotsc , \varepsilon_{m}, \delta_{1}, \dotsc , \delta_{n}$.   Let $\Fcat=\Fcat (\fg )$ denote the category of all integrable finite dimensional $\fg$-supermodules and all (not necessarily grading preserving) $\fg$-supermodule homomorphisms.  We should remark that our $\Fcat$ is a full subcategory of the category $\Fcat (\fg , \fg_{\0})$ considered in \cite{BKN1}.  However, the projective cover in $\Fcat (\fg , \fg_{\0})$ of any object of $\Fcat$ lies in $\Fcat$.  This implies that projective resolutions, cohomology, support varieties, etc.\ in the two categories coincide.   By $\fg$-supermodule we will always mean an object in $\Fcat$ unless otherwise stated.   Let $\Fcat_{\0}$ denote the category of all finite dimensional integrable $\fg$-supermodules and all grading preserving $\fg$-supermodule homomorphisms.  In \cite{serganova4} Serganova considers a full subcategory of $\Fcat_{\0}$.  However, the parity change functor allows us to apply her results without loss to $\Fcat_{\0}$ and $\Fcat$.

We fix the same choice of Borel subalgebra $\fb$ containing $\fh$ as in \cite{GS} and define $\rho$ to be the half sum of the positive even roots minus the half sum of the positive odd roots.  See just before \cite[Corollary 3]{GS} for a list of the simple roots and the $\rho$ corresponding to this choice of $\fb$.  Then the simple objects of $\Fcat$ are parameterized by highest weight with respect to our choice of $\fh$ and $\fb$.  We write $L(\lambda)$ for the simple supermodule of highest weight $\lambda \in \fh^{*}$.  By definition, we call $\lambda \in \fh^{*}$ a dominant integral weight if it is the highest weight of some simple supermodule in $\Fcat$.  For an explicit description of the dominant integral highest weights with respect to these choices, see \cite[Corollary 3]{GS}.

The maximal number of pairwise orthogonal isotropic roots with respect to the bilinear form on $\fh^{*}$ is the \emph{defect} of $\fg$.  We write $\defect(\fg )$ for the defect of $\fg$. In our case 
\[
\defect (\gl (m|n) )=\defect (\osp (2m|2n) )=\defect (\gl (2m+1|2n) ) = \min (m,n).
\]  
 The \emph{atypicality} of the simple supermodule $L(\lambda)$ is defined to be the maximal number of pairwise orthogonal isotropic roots which are also orthogonal to $\lambda + \rho$.  We write $\atyp (\lambda)$ for this number.  We write 
\begin{equation}\label{E:atypset}
A(\lambda) = \left\{\alpha_{1}, \dotsc , \alpha_{\atyp (\lambda)} \right\},
\end{equation}
for a fixed choice of such roots.  Although the set $A(\lambda)$ is not unique, it is known that the size of the set is well defined and, furthermore, does not depend on our choice of Cartan or Borel subalgebras.  Consequently, it makes sense to write $\atyp(L)$ for a simple supermodule $L$.  

\subsection{}    Given a $\fg$-supermodule $M = M_{\0} \oplus M_{\1}$ the \emph{superdimension} of $M$ is given by 
\[
\sdim (M) = \dim (M_{\0}) - \dim (M_{\1}).
\]

 Note that by definition we have $\atyp (L) \leq \defect (\fg )$ for any simple $\fg$-supermodule $L$.  The following conjecture of Kac and Wakimoto \cite[Conjecture 3.1]{KW} makes precise when we in fact have equality.
\begin{conj}\label{C:KWConj}  Let $L$ be a simple $\fg$-supermodule in $\Fcat$, then 
\[
\atyp (L) = \defect (\fg ) \text{ if and only if } \sdim (L) \neq 0.
\]
\end{conj}  While investigating generalized trace and dimension functions on nonsemisimple tensor categories, Geer, Patureau-Mirand and the author gave a generalized Kac-Wakimoto conjecture \cite[Conjecture 6.3.2]{GKP} (see Conjecture~\ref{C:GenKW}). Serganova recently proved the ordinary Kac-Wakimoto conjecture for $\gl (m|n)$ and $\osp (m|2n)$ and the generalized Kac-Wakimoto conjecture for $\gl (m|n)$ \cite{serganova4}.   Our first goal is to prove the generalized Kac-Wakimoto conjecture for  $\osp (m|2n)$.

\subsection{}\label{SS:gentrace}  In \cite{GKP} generalized trace and dimension functions on ribbon categories were introduced.  We only provide the definitions and results we need and refer the reader to \emph{loc.\  cit.} for additional details.  In order to be mathematically correct we work in $\Fcat_{\0}$ for the remainder of this section and Section~3 so as to have a ribbon category.  However simple arguments using the parity change functor show the results also hold in $\Fcat$.  We leave this to the interested reader.

For any object $V$ in $\Fcat_{\0}$, let $V^{*}$ denote the dual supermodule.  For any $V,W$ in $\Fcat_{\0}$, let $V \otimes W$ denote the tensor product supermodule (where the tensor product is over $\C$).  For any object $V$ in $\Fcat_{\0}$, let $\ev_{V}: V^{*} \otimes V \to \C$ be the evaluation morphism given by $f \otimes x \mapsto f(x)$ and let $\coev_{V}: \C  \to V \otimes V^{*}$ be the coevaluation morphism given by $1 \mapsto \sum_{i=1}^{n} v_{i} \otimes f_{i}$, where $v_{1}, \dotsc , v_{n}$ is a homogeneous basis for $V$ and where $f_{i} \in V^{*}$ is defined by $f_{i}(v_{j}) = \delta_{i,j}$.  Define the graded ``flip'' map $c_{V,W}: V \otimes W \to W \otimes V$ by $v \otimes w \mapsto (-1)^{\bar{v}\cdot \bar{w}}w \otimes v$.  Finally, for all $V$ in $\Fcat_{\0}$ we set the ``twist map'' $\theta_{V}: V \to V$ to be the identity.  The above data makes $\Fcat_{\0}$ into a \emph{ribbon category}.

For short we write $\ev'_{V} = \ev_{V} \circ c_{V,V^{*}}$ and $\coev'_{V} = c_{V,V^{*}} \circ \coev_{V}$.
Fix a pair of objects $V$ and $W$ in $\Fcat_{\0}$ and an endomorphism $f$ of $V\otimes W$.  To such objects and morphisms we use the ribbon category structure to define the following morphisms:
\begin{equation*}
\tr_{L}(f)=(\ev_{V}\otimes \Id_{W})\circ(\Id_{V^{*}}\otimes
f)\circ(\coev'_{V}\otimes \Id_{W}) \in \End_{\Fcat_{\0}}(W),
\end{equation*}
and
\begin{equation*}
\tr_{R}(f)=(\Id_{V}\otimes \ev'_{W}) \circ (f \otimes \Id_{W^{*}})
\circ(\Id_{V}\otimes \coev_{W}) \in \End_{\Fcat_{\0}}(V).
\end{equation*}

Given an object $J$ in $\Fcat_{\0}$, the \emph{ideal} $\ideal_{J}$ is the full subcategory of all objects which appear as direct summands of $J \otimes X$ for some object $X$ in $\Fcat_{\0}$.  More precisely, $M$ is an object of $\ideal_{J}$ if and only if there is an object $X$ in $\Fcat_{\0}$ and morphisms $\alpha: M \to J\otimes X$ and $\beta: J \otimes X \to M$ with $\beta \circ \alpha = \Id_{M}$.   For example, if $P$ is a projective supermodule in $\Fcat_{\0}$, then $\ideal_{P}$ is precisely the full subcategory of projective objects.  For short we denote this particular ideal by $\Proj$.

If $\ideal_{J}$ is an ideal in $\Fcat_{\0}$ then a \emph{trace on $\ideal_{J}$} is a family of linear functions
$$\mt = \{\mt_V:\End_{\Fcat_{\0}}(V)\rightarrow \C \}$$
where $V$ runs over all objects of $\ideal_{J}$ and such that following two conditions hold.
\begin{enumerate}
\item  If $U\in \ideal_{J}$ and $W$ is an object of $\Fcat_{\0}$, then for any $f\in \End_{\Fcat_{\0}}(U\otimes W)$ we have
\begin{equation}\label{E:VW}
\mt_{U\otimes W}\left(f \right)=\mt_U \left( \tr_R(f)\right).
\end{equation}
\item  If $U,V\in \ideal$ then for any morphisms $f:V\rightarrow U $ and $g:U\rightarrow V$  in $\Fcat_{\0}$ we have 
\begin{equation}\label{E:fggf}
\mt_V(g\circ f)=\mt_U(f \circ g).
\end{equation} 
\end{enumerate}  

 For $V$ an object of $\Fcat_{\0}$, we say a linear function
  $\mt:\End_{\Fcat_{\0}}(V)\rightarrow K$ is an \emph{ambidextrous trace on $V$} if for
  all $f\in \End_{\Fcat_{\0}}(V\otimes V)$ we have
$$\mt(\tr_L(f))=\mt(\tr_R(f)).$$   For short we call a supermodule \emph{ambidextrous} if it is simple and if it admits a nonzero ambidextrous trace.  The following theorem summarizes several results from \cite[Section 3.3]{GKP} as they apply here.

\begin{theorem}\label{T:ambi}  Let $L$ be a simple $\fg$-supermodule. If $\ideal_{L}$ admits a trace then the map $\mt_{L}$ is an ambidextrous trace on $L$. Conversely, an ambidextrous trace on $L$ extends uniquely to a trace on $\ideal_{L}$.  Furthermore, the trace on $\ideal_{L}$ and the ambidextrous trace on $L$ are unique up to multiplication by an element of $\C$.
\end{theorem}

Given a trace on $\ideal_{J}$, $\{\mt_{V} \}_{V \in \ideal_{J}}$, we define the modified dimension function on objects of $\ideal_{J}$,  
\[
\md_{J}: \obj(\ideal_{J}) \to \C,
\]
by taking the modified trace of the identity morphism:
\begin{equation*}
\md_{J}\left(V \right) =\mt_{V}(\Id_{V}).
\end{equation*}
We can now state the generalized Kac-Wakimoto conjecture \cite[Conjecture 6.3.2]{GKP}.

\begin{conj}\label{C:GenKW}  Let $\fg$ be a basic classical Lie superalgebra and let $J$ be a simple $\fg$-supermodule.  Then $J$ is ambidextrous and if $L$ is another simple supermodule with 
\[
\atyp (L) \leq \atyp(J),
\] then $L$ is an object of $\ideal_{J}$ and
\[
\atyp (L) = \atyp (J ) \text{ if and only if } \md_{J} (L) \neq 0.
\]
\end{conj} \noindent  This conjecture was proven for $\gl (m|n)$ by Serganova in \cite{serganova4}.

\section{Generalized Kac-Wakimoto Conjecture}\label{S:GKW}

\subsection{}\label{SS:fibrefunctor}   We first prove that every simple $\fg$-supermodule in $\Fcat_{\0}$ is ambidextrous.  To do so we use the fibre functor introduced by Duflo and Serganova \cite{DS} and further developed by Serganova \cite{serganova4}. We first summarize the results of theirs which we require.  

Let $G_{\0}$ denote the connected reductive algebraic group with Lie algebra $\fg_{\0}$.  Let 
\[
\Xvar = \left\{x \in \fg_{\1} \mid [x,x]=0 \right\} .
\]
Given an element $x \in \Xvar$, the $G_{\0}$-orbit of $x$ contains elements of the form $x_{1}+\dotsb + x_{k}$, where $x_{i}$ lies in the root space $\fg_{\alpha_{i}}$ and $\alpha_{1}, \dotsc , \alpha_{k}$ are pairwise orthogonal, isotropic roots. It is straightforward to see that the number $k$ depends only on the orbit and so it makes sense to define the \emph{rank} of $x$ to be $k$.  We write $\rank (x)$ for this number.  By definition the rank of $x$ is among $0, 1, \dotsc , \defect (\fg )$.  Using the root space decomposition of $\fg$ it is not difficult to see that every possible value is achieved.

Given an $x \in \Xvar$, let $\operatorname{Cent}_{\fg}(x)$ denote the centralizer of $x$ in $\fg$ and set 
\[
\fg_{x} = \operatorname{Cent}_{\fg}(x)/[x,\fg ].
\]  Note that $\fg_{x}$ is a Lie superalgebra and if $y \in \Xvar$ with $\rank(x) = \rank (y)$, then $\fg_{x} \cong \fg_{y}$.  Furthermore, if $\rank(\fg ) = k$, then we have: 
\begin{itemize}
\item if $\fg$ is $\gl  (m|n)$, then $\fg_{x}$ is isomorphic to $\gl (m-k|n-k)$;
\item  if $\fg$ is  $\osp (2m+1|2n)$, then $\fg_{x}$ is isomorphic to $\osp (2(m-k)+1|2(n-k))$;
\item  if $\fg$ is   $\osp (2m|2n)$, then $\fg_{x}$ is isomorphic to $\osp (2(m-k)|2(n-k)$;
\end{itemize}

If $x \in \Xvar$, then in the enveloping superalgebra of $\fg$ we have $0=[x,x]=2x^{2}$.  Hence for any supermodule $M$ in $\Fcat_{\0}$ the linear map $M \to M$ given by action of $x$ squares to zero.  That is, it makes sense to define 
\[
M_{x} = \operatorname{Ker}(x)/\operatorname{Im}(x).
\]  Note that $M_{x}$ is naturally a $\fg_{x}$-supermodule and the assignment $M \mapsto M_{x}$ defines a functor from  $\Fcat_{\0} $ to $\Fcat_{\0}(\fg_{x})$ which is called the \emph{fibre functor}.  We write $f \mapsto f_{x}$ for the functor's action on a morphism $f$.  Note that the fibre functor is a functor of ribbon categories.

\subsection{}  Recall that $\Proj$ is the ideal of all projective objects.  By \cite[Theorem 4.8.2]{GKP2} the ideal $\Proj$ is known to admit a nontrivial trace.  We write 
\[
\mt^{p} = \left\{\mt^{p}_{V} \mid V \in \Proj \right\}
\]
for this trace.  We call a simple supermodule \emph{typical} if it has atypicality zero.  By \cite[Theorem 1]{Kac2} if a simple supermodule is typical then it is projective.  Furthermore if $T$ is a typical simple supermodule, since $\Proj$ admits a nontrivial trace and $\ideal_{T}=\Proj$, it follows from Theorem~\ref{T:ambi} that $\mt^{p}_{T}(\Id_{T}) \neq 0$.

\begin{theorem}\label{T:EverySimpleIsAmbi} Let $\fg$ denote $\gl (m|n)$, $\osp (2m+1|2n)$, or $\osp (2m|2n)$.  Then every simple $\fg$-supermodule in $\Fcat_{\0}$ is ambidextrous.
\end{theorem}

\begin{proof} Let $L$ be a simple supermodule in $\Fcat_{\0}$ of atypicality $k$.   Fix $x \in \Xvar$ with $\rank(x)=k$.  For any $\fg$-supermodule $M$ let
\[
\varphi_{x}:\End_{\fg}\left(M \right) \to \End_{\fg_{x}}\left(M_{x} \right) 
\] denote the algebra map induced by the fibre functor via $\varphi_{x}(f) = f_{x}$.   By \cite[Corollary 2.2]{serganova4}  $L_{x}$ is a direct sum of typical supermodules and so is projective.  More generally, if $M$ is an object of $\ideal_{L}$, then it is a direct summand of $L \otimes Y$ for some supermodule $Y$.  Applying the fibre functor we see that $M_{x}$ is a direct summand of $L_{x} \otimes Y_{x}$ and so is projective.  Consequently it makes sense for any $M$ in $\ideal_{L}$ to define a map, $\mt_{M}$, by the composition 
\begin{equation}\label{E:mtracedef}
\mt_{M}:= \mt^{p}_{M_{x}}\circ \varphi_{x} : \End_{\Fcat_{\0}}\left(M \right) \to \C .
\end{equation}
Since the fibre functor is a functor of ribbon categories it is straightforward to verify that $\mt = \left\{\mt_{M} \mid V \in \ideal_{L} \right\}$ is a (possibly trivial) trace on the ideal $\ideal_{L}$.

We now prove  $\mt$ is nontrivial.  First we assume $\fg$ is either $\gl (m|n)$ or $\osp(2m+1|2n)$.  By \cite[Corollary 2.2]{serganova4}, since $L$ is a simple $\fg$-supermodule of atypicality $k$, we have 
\begin{equation}\label{E:Lxdef}
L_{x} \cong T \otimes C_{x}(L)
\end{equation}
as $g_{x}$-supermodules, where $T$ is a typical simple $\fg_{x}$-supermodule and $C_{x}(L)$ is a superspace with trivial $\fg_{x}$-action.  Using \eqref{E:Lxdef} we compute $\mt_{L}\left( \Id_{L}\right)$:
\[
\mt_{L}(\Id_{L}) = \mt^{p}_{L_{x}}(\Id_{L,x}) = \mt^{p}_{L_{x}}(\Id_{L_{x}})= \mt^{p}_{T}\left(\tr_{R}(\Id_{L_{x}}) \right) = \mt^{p}_{T}\left(\Id_{T} \right)\sdim \left(C_{x}(L) \right).
\] The first equality is by the definition of $\mt$, the second by the definition of the fibre functor, the third is by  \eqref{E:fggf} and \cref{E:Lxdef}, and the last is by direct calculation.   Furthermore, since $T$ is a typical simple supermodule we have $\mt^{p}_{T}(\Id_{T}) \neq 0$ and by \cite[Theorem 2.3]{serganova4} we have $\sdim \left(C_{x}(L) \right) \neq 0$.  Therefore $\mt_{L}$ is nontrivial and so $\mt$ is a nontrivial trace on $\ideal_{L}$ and $L$ is ambidextrous. 

The case when $\fg = \osp(2m|2n)$ is argued similarly.  The only difference is that by \cite[Corollary 2.2]{serganova4} we instead have
\begin{equation*}
L_{x} \cong T \otimes C'_{x}(L) \oplus T^{\sigma} \otimes C''_{x}(L)
\end{equation*}
as $\fg_{x}$-supermodules.   Here $T$ is a typical simple $\fg_{x}$-supermodule, $T^{\sigma}$ is the typical simple supermodule obtained by twisting $T$ by the involution $\sigma: \fg_{x}  \to \fg_{x}$ given just before \cite[Corollary 2.2]{serganova4}, and $C'_{x}(L)$ and $C''_{x}(L)$ are superspaces with trivial $\fg_{x}$-action.  
Using linearity and calculating as before, we have 
\begin{equation}\label{E:mtcalc}
\mt_{L}\left(\Id_{L} \right) = \mt^{p}_{T}\left(\Id_{T} \right) \sdim \left( C'_{x}(L) \right) +  \mt^{p}_{T^{\sigma}}\left(\Id_{T^{\sigma}}\right) \sdim \left( C''_{x}(L) \right).
\end{equation}

We now claim that $\mt^{p}_{T}\left(\Id_{T} \right) = \mt^{p}_{T^{\sigma}}\left(\Id_{T^{\sigma}}\right)$.  Define an endofunctor of $\Fcat_{\0}$ by twisting by $\sigma$:  on objects the functor is given by $M \mapsto M^{\sigma}$ and is the identity on morphisms.   Twisting by $\sigma$ is a functor of ribbon categories and takes $\Proj$ to itself.  Thus we may define a new family of maps $\mt^{\sigma} = \left\{\mt_{V}^{\sigma} \mid V \in \Proj  \right\}$ on $\Proj$ by precomposing by this functor: 
\[
\mt^{\sigma}_{M}(f) = \mt^{p}_{M^{\sigma}}(f).
\]  The fact that twisting by $\sigma$ is a functor of ribbon categories implies that $\mt^{\sigma}$ is a trace on $\Proj$.   Using this new trace we can rewrite our claim as $\mt^{p}_{T}\left(\Id_{T} \right) = \mt^{\sigma}_{T}\left(\Id_{T}\right)$.

Thus to prove our claim it suffices to prove that the traces $\mt^{p}$ and $\mt^{\sigma}$ coincide.  That is, that $\mt^{p}_{V}= \mt^{\sigma}_{V}$ for all $V$ in $\Proj$. By the explicit description of $\sigma$ given in \cite[Section 2]{serganova4} there exist typical simple supermodules $U$ for which $U^{\sigma}=U$  and for such a supermodule it is immediate that $\mt^{p}_{U} = \mt^{\sigma}_{U}$.   

However, $\ideal_{U}= \Proj$ and so by Theorem~\ref{T:ambi} a trace on $\Proj$ is completely determined by $\mt^{p}_{U}$.  That is, since  $\mt^{p}_{U}= \mt^{\sigma}_{U}$, we in fact have that $\mt^{p}$ and $\mt^{\sigma}$ coincide on all of $\Proj$.  In particular,  $\mt^{p}_{T}\left(\Id_{T} \right) = \mt^{\sigma}_{T}\left(\Id_{T}\right)$ and so $\mt^{p}_{T}\left(\Id_{T} \right) = \mt^{p}_{T^{\sigma}}\left(\Id_{T^{\sigma}}\right)$.  

Returning to \eqref{E:mtcalc}, we obtain
\begin{equation}\label{E:mtcalc2}
\mt_{L}\left(\Id_{L} \right) = \mt^{p}_{T}\left(\Id_{T} \right) \left[  \sdim \left( C'_{x}(L) \right) +  \sdim \left( C''_{x}(L) \right)\right].
\end{equation}  By \cite[Theorem 2.3]{serganova4} we have 
\begin{equation*}
\sdim \left( C'_{x}(L) \oplus C''_{x}(L) \right) =\sdim \left( C'_{x}(L)\right) + \sdim \left(C''_{x}(L) \right) \neq 0.
\end{equation*} Furthermore $T$ is a simple object in $\Proj$ and so $\mt^{p}_{T}(\Id_{T}) \neq 0$.  Combining these observations with \eqref{E:mtcalc2} we see that $\mt_{L}$ is nontrivial.  That is, $\mt$ defines a nontrivial trace on $\ideal_{L}$ and $L$ is ambidextrous. 
\end{proof}

We remark that the ambidextrous trace on $L$ given in the proof may depend on the choice of $x$.  However, since $L$ is simple any two traces differ only by a scalar multiple.  We also remark that our reduction to the typical case is inspired by Serganova's analogous approach for $\gl (m|n)$ given in \cite{serganova4}.  However, Serganova used a different argument to prove the nontriviality of the trace on $\ideal_{L}$.  Her approach uses the explicit description of the trace on typical supermodules given via supercharacters in \cite{GP0}. 

\subsection{}\label{SS:ideals} 
When $k=0$, we set $S_{0}$ to be a typical simple supermodule.  Then $\ideal_{S_{0}}=\Proj$ and it contains every typical simple supermodule and has a nontrivial trace by \cite[Theorem 4.8.2]{GKP2}.  By \cite[Lemma 6.3]{serganova4}, for each $0 < k \leq \defect (\fg )$, there exists a simple $\fg$-supermodule, $S_{k}$, of atypicality $k$ such that every simple supermodule of atypicality $k$ lies in $\ideal_{S_{k}}$.  By the previous theorem $S_{k}$ is ambidextrous and, in particular, we may fix an $x \in \Xvar$ of rank $k$ which defines a nontrivial trace on $\ideal_{S_{k}}$.  In either case we denote the trace on $\ideal_{S_{k}}$ by $\mt = \left\{\mt_{V} \mid V \in \ideal_{S_{k}} \right\}$ and the corresponding dimension function by $\md_{S_{k}}$.

\begin{prop}\label{T:onedirection}  Let  $0 \leq k \leq \defect (\fg )$ and let $S_{k}$ be the simple $\fg$-supermodule given above.  Let $L$ be a simple supermodule of atypicality $k$.  Then $\md_{S_{k}}(L) \neq 0$ and $\ideal_{L}=\ideal_{S_{k}}$.
\end{prop}

\begin{proof}  Since $L$ lies in $\ideal_{S_{k}}$, we have $\ideal_{L} \subseteq \ideal_{S_{k}}$.    By \cite[Theorem 4.2.1]{GKP} the nonvanishing of $\md_{S_{k}}(L)$ implies that the ideals are equal.  Thus it suffices to compute $\md_{S_{k}}(L)$. If $k=0$, then this is a consquence of Theorem~\ref{T:ambi} and  the fact that $\ideal_{L}=\Proj$.  
If $k>0$, then by the previous theorem $L$ is ambidextrous and using the element $x$ fixed above we also have a nontrivial trace on $\ideal_{L}$.  We denote this trace by $\mt' = \left\{\mt'_{V} \mid V \in \ideal_{L} \right\}$.  Recall in particular that $\mt'_{L}(\Id_{L}) \neq 0$.  Using the definition of $\mt$ and $\mt '$ we have 
\begin{equation*}
\md_{S_{k}}(L)  = \mt_{L}(\Id_{L}) = \mt_{L_{x}}^{p}(\Id_{L,x}) = \mt'_{L}\left(\Id_{L} \right)  \neq 0.
\end{equation*}
\end{proof}

It is worth making explicit the following representation theoretic interpretation of the previous result.
\begin{corollary}\label{C:corollary}  Let $\fg$ denote $\gl (m|n)$ or $\osp (m|2n)$.  Let $L_{1}$ and $L_{2}$ be two simple supermodules in $\Fcat_{\0}$ with the same atypicality.  Then there are supermodules $X_{1}$ and $X_{2}$ in $\Fcat_{\0}$ such that $L_{2}$ is a direct summand of $L_{1}\otimes X_{1}$ and  $L_{1}$ is a direct summand of $L_{2}\otimes X_{2}$. \end{corollary}

By the previous theorem for each $ 0 \leq k \leq \defect(\fg)$, the ideal generated by a simple supermodule of atypicality $k$ is independent of the choice of simple supermodule.  Consequently, we write $\ideal_{k}$ for the ideal generated by a simple of atypicality $k$.  In particular, $\ideal_{0} = \Proj$ (as typical supermodules are projective) and $\ideal_{\defect(\fg )} = \Fcat_{\0}$ (as the trivial supermodule has atypicality equal to the defect and generates the entire category). Furthermore it is not difficult to see using the translation functors of \cite[Sections 5-6]{GS} that for each atypicality $k = 1, \dotsc , \defect (\fg )$, there is a simple supermodule of atypicality $k$, $L$, and simple supermodule of atypicality $k-1$, $L'$, so that $L'$ is an object in $\ideal_{L}$.  Thus we have 
\[
\ideal_{0} \subseteq \ideal_{1} \subseteq \dotsb \subseteq \ideal_{\defect (\fg)}.
\]  

Given $x \in \Xvar$ of rank $k$, we write $\mt$ for the trace on $\ideal_{k}$ defined by \eqref{E:mtracedef} and $\md$ for the corresponding modified dimension function.

\begin{prop}\label{T:onedirection2}  Let  $0 < k \leq \defect (\fg )$ and let $L$ be a simple supermodule of atypicality strictly less than $k$.  Then $L$ is an object of $\ideal_{k}$ and $\md(L) = 0$ and $\ideal_{L} \subsetneq \ideal_{k}$.
\end{prop}

\begin{proof}  The fact that $L$ is an object of $\ideal_{k}$ follows from the discussion preceeding the proposition. We now compute $\md (L)$ using the definition of $\mt$ on $\ideal_{k}$. Since $L$ has atypicality strictly less than $k$ we know by  \cite[Theorem 2.1]{serganova4} that $L_{x}=0$.  It is then immediate that $\md (L)=0$.  Since $L$ is an object of $\ideal_{k}$, we have $\ideal_{L} \subseteq \ideal_{k}$.  However, by \cite[Theorem 4.2.1]{GKP} the vanishing of the modified dimension implies that the inclusion is strict.
\end{proof}

Combining the above results we have Conjecture~\ref{C:GenKW}.   We  also have the following description of the ideals defined by simple objects.
\begin{theorem}\label{T:chainofideals}  If $\ideal_{k}$ denotes the ideal defined by a simple supermodule of atypicality $k$ in $\Fcat_{\0}$, then $\ideal_{k}$ is independent of this choice.  Furthermore, these ideals form the following chain of inclusions 
\[
\Proj = \ideal_{0} \subsetneq \ideal_{1} \subsetneq \ideal_{2}\subsetneq \dotsb \subsetneq  \ideal_{\defect(\fg )} = \Fcat_{\0}.
\]
\end{theorem}

\section{Support Varieties}\label{S:supports}  

\subsection{}\label{SS:GenKWapp}  Given a classical Lie superalgebra $\fa$ and an object $M$ in $\Fcat (\fa)$, let $\mathcal{V}_{(\fa, \fa_{\0})}(M)$  denote the support variety of $M$ as defined in \cite{BKN1} and let $c_{\Fcat(\fa)}(M)$ the complexity of $M$ in $\Fcat (\fa)$ (i.e.\ the rate of growth of a minimal projective resolution of $M$ in $\Fcat$).  As an application of the generalized Kac-Wakimoto conjecture we see that for a simple supermodule these depend only on atypicality.

\begin{theorem}\label{T:constantatypicality}  Let $\fg$ denote $\gl (m|n)$ or $\osp (m|2n)$ and let $L_{1}$ and $L_{2}$ be two simple objects of $\Fcat (\fg )$ of the same atypicality.  Let $\fa \subseteq \fg$ denote a subalgebra of $\fg$ which is itself a classical Lie superalgebra.  Then 
\begin{align}\label{E:constantonatyp}
\mathcal{V}_{(\fa , \fa_{\0})}(L_{1}) &= \mathcal{V}_{(\fa , \fa_{\0})}(L_{2}) \\
c_{\Fcat(\fa )}(L_{1}) &= c_{\Fcat(\fa )}(L_{2}) 
\end{align}
\end{theorem} 
\begin{proof}  By Corollary~\ref{C:corollary} there is a $\fg$-supermodule $X$ such that $L_{1}$ is a direct summand of $L_{1} \otimes X$.  By the basic properties of support varieties \cite[Equations (4.6.3) and (4.6.4)]{BKN2} this implies 
\[
\mathcal{V}_{(\fa , \fa_{\0})}(L_{1}) \subseteq \mathcal{V}_{(\fa , \fa_{\0})}(L_{2} \otimes X) \subseteq \mathcal{V}_{(\fa , \fa_{\0})}(L_{2}) \cap \mathcal{V}_{(\fa , \fa_{\0})}(X)  \subseteq \mathcal{V}_{(\fa , \fa_{\0})}(L_{2}).
\]  However, this argument is symmetric in $L_{1}$ and $L_{2}$ and so we have the equality of support varieties.

 To prove equality of complexity, we observe that the argument used for $\gl(m|n)$ in the proof of \cite[Theorem 8.1.1]{BKN4} applies verbatim with the exception that references to \cite[Corollary 6.7]{serganova4} should be replaced with references to the generalized Kac-Wakimoto conjecture. 
\end{proof}

\subsection{}\label{SS:blockEq}  Using Theorem~\ref{T:constantatypicality} and the line of argument for $\gl (m|n)$ used in \cite{BKN2}, we now compute the support varieties for the simple supermodules of $\Fcat$.  If $L(\lambda)$ is typical, then it is projective by \cite[Theorem 1]{Kac2} and the support variety is trivial.  Theorem~\ref{T:typeA} immediately follows.  Consequently we assume $\atyp (L(\lambda)) >0$ in what follows.  

 Given $0 < k \leq \defect (\fg)$, let $\fg_{k}$ be the subalgebra of $\fg$ defined as follows:
\begin{itemize}
\item for $\fg =\gl (m|n)$, $\fg_{k}= \gl (k|k)$,
\item for $\fg = \osp (2m+1|2n)$,  $\fg_{k}= \osp (2k+1|2k)$,
\item for $\fg = \osp(2m|2n)$, $\fg_{k}= \osp (2k|2k)$.
\end{itemize} 

We identify $\fg_{k}$ as a subalgebra of $\fg$ as follows.  For $\gl(m|n)$, $\fg_{k}$ is the subalgebra isomorphic to $\gl (k|k)$ whose roots lie in the intersection of $\Phi$ with the $\mathbb{R}$-span of $\varepsilon_{m-k+1}, \dotsc , \varepsilon_{m}, \delta_{1}, \dotsc , \delta_{k}$.  Similarly, for $\osp (2m|2n)$ and $\osp (2m+1|2n)$, $\fg_{k}$ is the subalgebra whose roots like in the intersection of $\Phi$ with the $\mathbb{R}$-span of $\varepsilon_{m-k+1}, \dotsc , \varepsilon_{m}, \delta_{n-k+1}, \dotsc , \delta_{n}$.  In particular, note that $\fg_{k}$ has defect $k$.

 Let $\mathcal{Z}=\mathcal{Z}(U(\fg))$ denote the center of the universal enveloping superalgebra of $\fg$.  Given a simple $\fg$-supermodule $L(\lambda)$ of highest weight $\lambda$ we may use Schur's lemma to define an algebra homomorphism $\chi_{\lambda}: \mathcal{Z} \to \C $ by the equation $zv = \chi_{\lambda}(z)v$ for all $z \in \mathcal{Z}$ and all $v \in L(\lambda)$.  Using these central characters we have a decomposition of $\Fcat$ into blocks 
\[
\Fcat = \bigoplus \Fcat^{\chi},
\] where the direct sum runs over all algebra homomorphisms $\chi :  \mathcal{Z} \to \C$.   It is known that all simple supermodules in $\Fcat^{\chi}$ have the same atypicality and so it makes to refer to this as the atypicality of the block.  In particular, the principal block of $\Fcat (\fg_{k})$ has atypicality $k$.  Gruson and Serganova prove that every block\footnote{More precisely, for $\osp (2m|2n)$ half the blocks of atypicality $k$ are equivalent to the principle block of $\osp (2k+2|2k)$.  See \cite[Section 5]{GS} for details.  However, by Theorem~\ref{T:constantatypicality} we may safely assume that our simple supermodule does not lie in one of these blocks. } of $\Fcat (\fg )$ of atypicality $k$ is equivalent to the principle block of $\Fcat (\fg_{k})$.  We will need to study the functor which gives this equivalence.

 Let $\fl$ denote the subalgebra $\fg_{k} + \fh \subseteq \fg$.  Fix a choice of $\fh ' \subset \fh$ so that $\fh '$ is a central subalgebra of $\fl$ and $\fl = \fg_{k}\oplus \fh'$.  Given $\lambda \in \fh^{*}$, let $\lambda' \in (\fh')^{*}$ denote the map obtained by restricting $\lambda$ to $\fh '$.  Given a dominant weight $\mu$, Gruson and Serganova call $\mu$ \emph{stable} if $A(\mu)$ (where $A(\mu)$ is as in \eqref{E:atypset}) is a subset of the roots for $\fl$ and $(\mu +\rho, \beta)>0$ for all $\beta \in \Phi_{\0}^{+}$ which are not roots of $\fl$.   Say $\lambda$ and $\mu$ are stable dominant integral weights and $\chi_{\mu}=\chi_{\lambda}$.  Then we have by \cite[Section 3]{GS} and references therein that $\mu$ can be written as $w(\lambda + \rho +\sum_{i}n_{i}\alpha_{i}) - \rho$ where $w \in  W$, the Weyl group of $\fg_{\0}$, the sum is over the elements of $A(\lambda)$, and $n_{i} \in \C$ for all $i$.   From this it follows that $\lambda' = \mu'$.     If $\mu$ is a stable dominant integral weight, then is straightforward to verify that on $L(\mu)$  the Gruson-Serganova functor given in \cite[Section 5]{GS} coincides with the functor $\Res_{\mu'}: \Fcat (\fg ) \to \Fcat (\fg_{k})$ given by
\begin{equation}\label{E:Resprime}
 \Res_{\mu'}(N)=\left\{n \in N \mid h'n = \mu' (h')n \text{ for all } h' \in \fh' \right\}.
\end{equation}  Note that this is indeed a $\fg_{k}$-supermodule as $\fh'$ commutes with $\fg_{k}$.  Let  $N$ be an object of $\Fcat^{\chi_{\mu}}$ such that for every composition factor $L(\gamma)$ of $N$, the weight $\gamma$ is stable.    An induction on composition series length using that $\gamma'=\mu'$ shows that the Gruson-Serganova functor coincides with $\Res_{\mu'}$ on $N$.

\subsection{}   The inclusion $\fg_{k} \hookrightarrow \fg $ induces a map in relative cohomology, 
\[
\res : \HH^{\bullet}(\fg , \fg_{\0 }; M) \to  \HH^{\bullet}(\fg_{k}, \fg_{k, \0 }; M),
\] for any $M$ in $\Fcat(\fg).$  Note that this coincides with the map induced by the restriction functor, $\Res :\Fcat (\fg) \to \Fcat (\fg_{k}).$  We then have the following commutative diagram. 

\begin{equation}\label{E:commute}
\begin{CD}
 I_{\fg }(M) \hookrightarrow \HH^{\bullet}(\fg, \fg _{\0 }; \C)  @>^{m_{1}}>>        \HH^{\bullet}(\fg, \fg_{ \0 }; M\otimes M^{*}) \\
@V\res_{\C} VV                                                           @VV\res V\\
  I_{\fg_{k}}(M) \hookrightarrow \HH^{\bullet}(\fg_{k}, \fg_{k, \0 }; \C)  @>^{m_{2}}>>        \HH^{\bullet}(\fg_{k}, \fg_{k, \0 }; M\otimes M^{*}) 
\end{CD}
\end{equation}
Here the horizontal maps are those induced by the exact functor $-\otimes M$, and $I_{\fg }(M)$ (resp.\ $I_{\fg_{k}}(M)$) is the kernel of this map. Recall that this is the ideal which defines $\mathcal{V}_{(\fg, \fg_{\0})}(M)$ (resp.\  $\mathcal{V}_{(\fg_{k}, \fg _{k, \0})}(M)$).  

For clarity in our notation we shall capitalize the names of functors and call the induced maps in cohomology by the same name but in lower case.  For example, in \eqref{E:commute} $\res_{\C}$ denotes the map induced by the restriction functor $\Res$ (with coefficients in the trivial supermodule).
Let $J$ denote the kernel of $\res_{\C}$.  Fix $ d \geq 0$ so that $J$ is generated by elements of degree no more than $d$.   Such a $d$ exists because $\HH^{\bullet}(\fg, \fg _{\0 }; \C)$ is a Noetherian ring (indeed by \cite[Theorem 4.1.1]{BKN1} it is a polynomial ring).  Now choose a dominant integral weight $\lambda$ and let $P_{\bullet} \to L(\lambda)$ be a fixed projective resolution of $L(\lambda)$ in $\Fcat (\fg )$.  We set $\Gamma$ to be the set of highest weights of the composition factors of $P_{0}, \dotsc , P_{d}$.  Applying the algorithm given in the proof of \cite[Lemma 12]{DS} we may choose $\lambda$ so that $\gamma$ is stable for all $\gamma \in \Gamma$.  

Let us write $\fe \subseteq \fg$ and $\tilde{\fe} \subseteq \fg_{k}$ for the detecting subalgebras as defined in \cite[Section 4]{BKN1}.  We may assume that $\tilde{\fe}_{\1 } \subseteq \fe_{\1 }$.  To see this, we see that one can choose a set $\Omega$ as in \cite[Table 2]{BKN1} to obtain an explicit basis for $\fe_{\1}$ for which $\tilde{\fe}_{\1 } \subseteq \fe_{\1 }$. The following proposition records certain properties of \eqref{E:commute} and is straightforward generalization of \cite[Proposition 4.7.3]{BKN2}.  For completeness we include the proof.

\begin{prop}\label{P:diagrams}    Let $J$ denote the kernel of the map $\res_{\C}$ and fix $ d \geq 0$ so that $J$ is generated by elements of degree no more than $d.$  Then the following statements about \eqref{E:commute} hold true.
\begin{enumerate}
\item [(a)] The map $\res_{\C}$ is a surjective algebra homomorphism.
\item  [(b)] Let $M=L(\lambda)$ be a simple supermodule in $\Fcat$ of atypicality $k.$  Then the map $m_{2}$ is injective. 
\item  [(c)] Assume $M=L(\lambda)$ be a simple supermodule in $\Fcat $ so that the elements of the set $\Gamma$ defined above are stable.  Then $J \subseteq I_{\fg }(M).$
\end{enumerate}
\end{prop}

\begin{proof} 

\noindent  One proves (a) as follows.     By \cite[Theorem 3.3.1(a)]{BKN1} there are finite pseudoreflection groups $\mathcal{W}$ and $\widetilde{\mathcal{W}}$ for which restriction induces isomorphisms  $ \HH^{\bullet}(\fg, \fg_{\0 }; \C) \to S(\fe_{\1}^{*})^{\mathcal{W}}$ and $ \HH^{\bullet}(\fg_{k}, \fg_{k, \0 }; \C)  \to  S(\tilde{\fe}_{\1}^{*})^{\widetilde{\mathcal{W}}}.$

From the identification $\widetilde{\fe}_{\1}\subseteq \fe_{\1},$ one has the canonical algebra homomorphism given by restriction of functions
\[
\rho: S(\fe^{*}_{\1})^{\mathcal{W}} \to S(\tilde{\fe}^{*}_{\1})^{\widetilde{\mathcal{W}}}.
\] and the explicit description of $\fe$ and $\widetilde{\fe}$ allows one to verify that this map is surjective.

As all maps are induced by restrictions, one has the following commutative diagram.
\begin{equation*}
\begin{CD}
 \HH^{\bullet}(\gl(m|n), \gl (m|n)_{\0 }; \C)  @>_{\simeq}>>       S(\fe_{\1}^{*})^{\mathcal{W}} \\
@V\res_{\C} VV                                                           @VV \rho  V\\
  \HH^{\bullet}(\gl(k|k), \gl (k|k)_{\0 }; \C)  @>_{\simeq}>>         S(\tilde{\fe}_{\1}^{*})^{\widetilde{\mathcal{W}}} 
\end{CD}
\end{equation*}  Therefore, the map $\res_{\C}$ is surjective.

To prove (b) one argues as follows. We first assume that $\lambda$ is stable.  We then have the decomposition
\begin{equation}\label{E:SimpleDecomp}
L(\lambda) = \Res_{\lambda'}(L(\lambda)) \oplus G_{\lambda'}(L(\lambda))
\end{equation}
as $\fg_{k}$-supermodules, where 
\begin{equation}\label{E:Gmudef}
G_{\lambda'}(L(\lambda))=\sum_{\substack{\nu \in (\mathfrak{h}')^{*}\\ \nu \neq \lambda'}}\{x \in L(\lambda) \mid hx = \nu(h)x \text{ for all } h \in \mathfrak{h}'\}. 
\end{equation}  Now since $\Res_{\lambda'}(L(\lambda))$ coincides with the output of the Gruson-Serganova equivalence it is a simple $\fg_{k}$-supermodule in the principle block of $\Fcat (\fg_{k})$.  That is, it is a simple supermodule of the same atypicality as the trivial $\fg_{k}$-supermodule.  By Theorem~\ref{T:constantatypicality}  this implies the following equality.  The remaining inclusions follow by the basic properties of support varieties: 
\[
\V_{(\fg_{k}, \fg_{k,\0})}\left( \C\right) = \V_{(\fg_{k}, \fg_{k,\0})}\left( \Res_{\lambda'}(L(\lambda))\right) \subseteq \V_{(\fg_{k}, \fg_{k,\0})}\left( L(\lambda)\right) \subseteq \V_{(\fg_{k}, \fg_{k,\0})}\left( \C\right).
\]  Therefore we have
\begin{equation}\label{E:equalvar}
\V_{(\fg_{k}, \fg_{k,\0})}\left(L(\lambda)\right) = \V_{(\fg_{k}, \fg_{k,\0})}\left( \C\right).
\end{equation}  Applying Theorem~\ref{T:constantatypicality} again, it follows that \eqref{E:equalvar} holds for arbitrary $L(\lambda)$ when $\lambda$ has atypicality $k$.  However $\HH^{\bullet}(\fg_{k}, \fg_{k,\0 }; \C)$ is a polynomial ring and so has no nonzero nilpotent elements.  This along with \eqref{E:equalvar} implies that $I_{\fg_{k}}(L(\lambda))=(0)$ and the injectivity of $m_{2}$ follows.

 We now prove $(c).$  By our assumption on $\Gamma$ the functor $\Res_{\lambda'}$ coincides with the Gruson-Serganova equivalence on the first $d$ degrees of cohomology and so $\res_{\lambda'}$ defines an isomorphism in cohomology in those degrees.  Let $\Res_{\fl}: \Fcat (\fg ) \to \Fcat (\fl )$ be the restriction functor and $P_{\lambda'}: \Fcat(\fl ) \to \Fcat (\fg_{k})$ be the functor given by projection onto the $\lambda'$ weight space with respect to the action of $\fh '$, then $\Res_{\lambda'} = P_{\lambda'} \circ \Res_{\fl}$.  Since  $\res_{\lambda'}$ is injective for $i=0, \dotsc , d$, $\res_{\fl}$ must also be injective in these degrees.  The fact that $\fl = \fg_{k}\oplus \fh '$ for a central abelian subalgebra $\fh '$ implies that the restriction functor $\Fcat (\fl) \to \Fcat (\fg_{k})$ induces an injective map on cohomology.   Composing the $\Res_{\fl}$ with this functor yields the restriction functor $\fg \to \fg_{k}$ and, hence, $\res$ is injective.  From this and the commutativity of the diagram \eqref{E:commute}, it follows that the generators of $J$ and hence $J$ itself lies in $\ideal_{\fg}(M)$.
\end{proof}

\subsection{}\label{SS:supports}  We can now compute the support varieties of the simple supermodules. Let $\fe \subseteq \fg$ be the detecting subalgebra of $\fg$.  Let $\mathcal{W}$ be the finite pseduoreflection groups given by \cite[Theorem 3.3.1(a)]{BKN1}. For any $\fg$-supermodule, $M$, the inclusion $\fe \hookrightarrow \fg$ induces a map of support varieties
\begin{equation}
\res^{*}: \V_{(\fe ,\fe_{\0})}(M) \to \V_{(\fg ,\fg_{\0})}(M)
\end{equation}
with image 
\begin{equation}\label{E:resmapsbetweenvarieties}
\res^{*}\left(  \V_{(\fe ,\fe_{\0})}(M)\right) \cong \V_{(\fe ,\fe_{\0})}(M) / \mathcal{W}.
\end{equation}

The proof of the following theorem closely parallels the analogous result in \cite{BKN2}.  We include the proof for completeness.
\begin{theorem}\label{T:typeA}  Let $\fg$ be $\gl (m|n)$ or $\osp (m|2n)$.  Let $L(\lambda)$ be a simple $\fg$-supermodule of atypicality $k.$  Let $\widetilde{\fe} \subseteq \fg_{k}$ be the detecting subalgebra of $\fg_{k}$ chosen so that $\widetilde{\fe}_{\1} \subseteq \fe_{\1}$.  Then, 

\begin{enumerate}
\item [(a)] 
\begin{equation}\label{E:gvariety}
\operatorname{res}^{*}(\widetilde{\fe}_{\1})=\operatorname{res}^{*}\left(\mathcal{V}_{(\fe, \fe_{\0})}\left( L(\lambda)\right) \right) = \mathcal{V}_{(\fg, \fg_{\0})}\left(L(\lambda) \right) \cong \mathbb{A}^{k}. 
\end{equation}
\item [(b)]  
\begin{equation}\label{E:evariety1}
\mathcal{V}_{(\fe ,\fe_{\0})}(L(\lambda)) = \mathcal{W} \cdot \widetilde{\fe}_{\1}.
\end{equation}
In particular, $\mathcal{V}_{(\fe ,\fe_{\0})}(L(\lambda))$ is the union of finitely many $k$-dimensional subspaces. 
\end{enumerate} 
\end{theorem}

\begin{proof} One proves (a) as follows.   By Theorem~\ref{T:constantatypicality} we may compute the support variety of any simple supermodule of atypicality $k$.  We choose $L(\lambda)$ so that the statements of Proposition~\ref{P:diagrams} hold true.  By Proposition~\ref{P:diagrams}(c) we have that 
$\operatorname{Ker}(\res_{\C}) \subseteq I_{\fg}(L(\lambda)).$   On the other hand, it follows by the commutativity of \eqref{E:commute} and the injectivity of $m_{2}$ (Proposition~\ref{P:diagrams}(b)) that $I_{\fg}(L(\lambda)) \subseteq \operatorname{Ker}(\res_{\C}).$  Therefore, $I_{\fg}(L(\lambda)) = \operatorname{Ker}(\res_{\C}).$  Using the surjectivity of $\res_{\C}$ and the description of $\HH^{\bullet}(\fg_{k}, \fg_{k, \0 }; \C))$ as a polynomial ring in $k$ variables, we have
\begin{align*}
\mathcal{V}_{(\fg, \fg_{\0})}(L(\lambda)) &\cong \operatorname{MaxSpec}\left(\HH^{\bullet}(\fg, \fg_{\0 }; \C)/\operatorname{Ker}(\res_{\C}) \right) \\
&\cong \operatorname{MaxSpec}\left(\HH^{\bullet}(\fg_{k}, \fg_{k,\0 }; \C) \right) \\
&\cong \mathbb{A}^{k}.
\end{align*}

Now consider $\mathcal{V}_{(\fe, \fe_{\0})}(L(\lambda)).$  Recall that $\tilde{\fe}_{\1} \subseteq \fe_{\1}.$  Since $L(\lambda)$ is stable it follows that $L(\lambda)$ contains a simple $\fg_{k}$-supermodule of atypicality $k$ as a direct summand (namely
$\Res_{\lambda'}L(\lambda)$) it follows by Theorem~\ref{T:constantatypicality} that $\mathcal{V}_{(\tilde{\fe}, \tilde{\fe}_{\0})}(L(\lambda)) = \mathcal{V}_{(\tilde{\fe}, \tilde{\fe}_{\0})}(\C)$.   By the rank variety description of $\widetilde{\fe}$ support varieties  it must be that for any $x \in \tilde{\fe}_{\1},$ $L(\lambda)$ is not projective as an $\langle x \rangle$-supermodule. Here $\langle x \rangle$ denotes the Lie subsuperalgebra generated by $x.$  
This statement is equally true if we view $x$ as an element of $\fe_{\1}.$  Thus, we have $\widetilde{\fe}_{\1} \subseteq \mathcal{V}_{(\fe , \fe_{\0})}(L(\lambda)).$ Therefore by \eqref{E:resmapsbetweenvarieties} one has,  
\begin{equation}\label{E:anotherdamnequation}
\res^{*}(\widetilde{\fe}_{\1}) \subseteq \res^{*}(\mathcal{V}_{(\fe, \fe_{\0})}(L(\lambda))) \subseteq \mathcal{V}_{(\fg, \fg_{\0})}(L(\lambda)) \cong \mathbb{A}^{k}.
\end{equation}
However, by \eqref{E:resmapsbetweenvarieties} the map $\res^{*}$ is finite-to-one so $\res^{*}\left( \widetilde{\fe}_{\1}\right)$ is a $k$-dimensional closed subset of $\mathbb{A}^{k}.$  However $\mathbb{A}^{k}$ is a $k$-dimensional irreducible variety.  Therefore $\res^{*}\left( \widetilde{\fe}_{\1}\right)=\mathbb{A}^{k}$ and all the containments in \eqref{E:anotherdamnequation} must be equalities.  This proves $(a).$
To prove $(b)$ we simply use the fact that the fibers of the map $\res^{*}$ are precisely the orbits of the finite group $\mathcal{W}$. 
\end{proof}

The above theorem immediately implies the validity of the atypicality conjecture for $\gl (m|n)$ and $\osp (m|2n)$.

\begin{corollary}\label{C:atypicalitycorollary}  Let $\fg$ denote $\gl (m|n)$ and $\osp (m|2n)$ and let $L(\lambda)$ be a simple $\fg$-supermodule of atypicality $k$.  Then
\[
\dim \V_{(\fe, \fe_{\0})}(L(\lambda)) = \dim \V_{(\fg, \fg_{\0})}(L(\lambda)) = \atyp (L(\lambda)).
\]
\end{corollary}

\linespread{1}

\end{document}